\documentclass[12pt]{article}
\usepackage[top=3cm, bottom=3cm, left=3.5cm, right=3.5cm]{geometry}
\usepackage{amsmath, amssymb, amsthm}
\usepackage{mathrsfs}

\newtheorem{thm}{Theorem}[section]

\newtheorem{lem}[thm]{Lemma}
\newtheorem{cor}[thm]{Corollary}

\theoremstyle{definition}
\newtheorem{dfn}{Definition}[section]

\newtheorem{ques}{Question} 

\theoremstyle{remark}
\newtheorem{rem}{Remark}

\everymath{\displaystyle}

\newcommand{\C}{\mathbb{C}}

\newcommand{\N}{\mathbb{N}}

\newcommand{\R}{\mathbb{R}}

\newcommand{\cC}{\mathcal{C}}

\newcommand{\cH}{\mathcal{H}}

\newcommand{\cW}{\mathcal{W}}

\newcommand{\fm}{\mathfrak{m}}

\newcommand{\Com}[2]{\left[{#1}, {#2}\right]}
\newcommand{\NrmPhi}[1]{   \mbox{$\left|{#1}\right|_{\Phi}$}  }
\newcommand{\NrmPhib}[1]{\mbox{$\left|{#1}\right|_{\widetilde{\Phi}}$}}
\newcommand{\Nrm}[1]{\left\|{#1}\right\|}

\newcommand{\kpb}[1]{\widetilde{k_p^-}\left(#1\right)}
\newcommand{\diag}{\operatorname{diag}}
\newcommand{\diam}{\operatorname{diam}}

\numberwithin{equation}{section}

\title{Quasicentral modulus and self-similar sets:\\ a supplementary result to Voiculescu's work}
\author{Kozo Ikeda  
and 
Masaki Izumi\\
Graduate School of Science \\
Kyoto University \\
Sakyo-ku, Kyoto 606-8502, Japan }

\begin{document}

\maketitle
\begin{abstract} In his recent work, Voiculescu generalized his remarkable formula for the quasicentral modulus 
of a commuting $n$-tuple of hermitian operators with respect to the $(n,1)$-Lorentz ideal to the case where its spectrum 
is contained in a Cantor-like self-similar set in a certain class. 
In this note, we treat general self-similar sets satisfying the open set condition, and obtain lower and upper bounds of the quasicentral modulus. 
Our proof shows that Voiculescu's formula holds for a class of self-similar sets including the Sierpinski gasket and 
the Sierpinski carpet. 
\end{abstract}

\section{Introduction}
Pursuing his idea in the proof of a non-commutative Weyl-von Neumann theorem for C$^*$-algebras \cite{V76}, 
D. Voiculescu introduced in \cite{V79} the quasicentral modulus $k_{\Phi}(\tau)$ for an $n$-tuple $\tau$ of Hilbert space operators 
with respect to a symmetric normed ideal $\mathfrak{S}_\Phi$. 
With this quantity, he established a far reaching generalization of the classical Weyl-von Neumann theorem in perturbation theory: 
an $n$-tuple of commuting hermitian operators $\tau$ is a $\mathfrak{S}_{\Phi}^{(0)}$-perturbation of a commuting tuple of 
diagonal hermitian operators if and only if $k_\Phi(\tau)$ vanishes. 
Moreover he showed that $k_n(\tau):=k_{\cC_n}(\tau)$ always vanishes for the Schatten-von Neumann ideal $\cC_n$ 
(see also \cite{BV}, \cite{DV}, \cite{V81}, \cite{V84}, \cite{V90} for later development). 

Besides the significance of whether it vanishes or not, the quantity $k_n^-(\tau):=k_{\cC_n^-}(\tau)$ itself is of particular interest for 
the $(n,1)$-Lorentz ideal $\cC_n^-$ because Voiculescu obtained the following remarkable formula:
 
\begin{thm}[{\cite[Theorem 4.5]{V79}}]
\label{thm kn- estimate} 
Let $\tau$ be an $n$-tuple of commuting hermitian operators, and let $\fm$ be the multiplicity function of the Lebesgue absolutely 
continuous part of $\tau$. Then we have
$$k_n^{-}(\tau)^n = \gamma_n \int_{\mathbb{R}^n} \fm(s) d\lambda(s), $$
where $0<\gamma_n<\infty$ is a constant independent of $\tau$ and $\lambda$ is the Lebesque measure.
\end{thm}

In his recent work \cite{V2021}, Voiculescu generalized the above result to the case where the joint spectrum $\sigma(\tau)$ 
is a subset of a Cantor-like self-similar set $K$ in a certain class and the Lebesgue measure is replaced by the Hausdorff measure 
(see Theorem \ref{A4} below).  
The purpose of this note is to treat a larger class of self-similar sets than Voiculescu did. 
Although we cannot show the precise equality as above, we obtain upper and lower bounds establishing that $k_p^-(\tau)^p$ 
behaves like the Hausdorff measure $\cH_p$ with $p$ the Hausdorff dimension of $K$ for general self-similar sets satisfying 
the open set condition (Theorem \ref{main}). 
Our proof shows that Voiculescu's formula holds for a class of self-similar sets including the Sierpinski gasket and 
the Sierpinski carpet (Remark \ref{Sierpinski}). 

The authors would like to thank Dan Voiculescu for useful discussion, and Kenneth Falconer and Jun Kigami for having informed 
the authors of the reference \cite{Mo}.

\section{Preliminaries}
Our basic references for normed ideals are \cite{GK} and \cite{S}, while we adopt 
the notation in \cite{V79} and \cite{V2021}.  

Throughout the paper $H$ denotes a separable Hilbert space. 
The symbols $\mathscr{L}(H), \mathscr{K}(H), \mathscr{P}(H), \mathscr{F}(H), \mathscr{R}(H)_1^{+}$ 
will respectively denote the bounded operators, the compact operators, the finite rank orthogonal projections, the finite rank operators, and the finite rank positive contractions on $H$. $\mathscr{P}(H)$ and $\mathscr{R}^{+}_1(H)$ are both directed sets by natural order of self-adjoint operators. 

For $T\in \mathscr{K}(H)$, we denote by $\{\mu_j(T)\}_{j=1}^\infty$ the singular value sequence of $T$, that is, the 
set of eigenvalues of $|T|$ in decreasing order. 
For a symmetric function $\Phi$, we denote by $|T|_{\Phi}$ the norm of the operator ideal corresponding to $\Phi$. 
For $1\leq p< \infty$, let 
$$|T|_p=\left(\sum_{j=1}^\infty\mu_j(T)^p\right)^{1/p},$$
and for $1<p\leq \infty$, let
$$|T|_p^-=\sum_{j=1}^\infty \frac{\mu_j(T)}{j^{1-\frac{1}{p}}}.$$
The Schatten-von Neumann ideal $\cC_p$ is the set of compact operators $T$ with finite $|T|_p$, and 
the $(p,1)$-Lorentz ideal $\cC_p^-$ is the set of compact operators $T$ with finite $|T|_p^-$. 
Then we have 
$$\bigcup_{r<p}\cC_r\subset \cC_p^-\subset \cC_p.$$

Let $\tau^{(j)} = (T_1^{(j)}, T_2^{(j)}, \ldots , T_n^{(j)}) \in (\mathscr{L}(H))^{n},\  \sigma = (S_1, S_2, \ldots , S_n) \in (\mathscr{L}(H))^{m}$, 
$X, Y \in \mathscr{L}(H)$, and $Z\in \mathscr{L}(K)$. 
We will write
\[
\tau^{(1)} \oplus \tau^{(2)} \equiv (T_1^{(1)} \oplus T_1^{(1)}, T_2^{(1)} \oplus T_2^{(2)}, \ldots , T_n^{(1)} \oplus T_n^{(2)}) \in \mathscr{L}(H \oplus H)^{n}
\]
\[
\tau^{(1)} + \tau^{(2)} \equiv (T_1^{(1)} + T_1^{(1)}, T_2^{(1)} + T_2^{(2)}, \ldots , T_n^{(1)} + T_n^{(2)}) \in \mathscr{L}(H)^{n}
\]
\[
X\tau Y = (XT_1Y, XT_2Y, \ldots , XT_nY) \in \mathscr{L}(H)^{n}
\]
\[
\tau \otimes Z = (T_1 \otimes Z, T_2 \otimes Z, \ldots, T_n \otimes Z) \in \mathscr{L}(H \otimes K)^n
\]
\[
\left [ X, \tau \right] = (\left[X, T_1\right], \left[X, T_2\right], \ldots , \left[X, T_n\right]) \in \mathscr{L}(H)^{n}
\]
\[
\|\tau \| = \max_{1 \le j \le n} \| T_j \|
\]
\[
|\tau |_{\Phi} = \max_{1 \le j \le n} | T_j |_{\Phi},\quad |\tau|_p=\max_{1 \le j \le n} | T_j |_p,\quad  
|\tau|_p^-=\max_{1 \le j \le n} | T_j |_p^-. 
\]

\begin{dfn}[{\cite[Section 1]{V79}}]
For $\tau \in \mathscr{L}(H)^n$, we define 
\[
k_\Phi(\tau) = \liminf_{A \in \mathscr{R}^+_1(H)} \NrmPhi{\Com{\tau}{A}}.
\]
When $|T|_{\Phi}=|T|_p$, we simply denote it by $k_p(\tau)$, and 
when $|T|_{\Phi}=|T|_p^-$, we denote it by $k_p^-(\tau)$. 
\end{dfn}

We collect basic facts frequently used in this paper (see {\cite[Section 1]{V79}}). 

\begin{thm}\label{basics} 
Let the notation be as above. 
The following hold:
\begin{itemize}
\item[$(1)$] If $\{A_i\}$ is a sequence in $\mathscr{R}_1^{+}(H)$ conversing to $I$ in the weak operator topology, then
\[
k_\Phi(\tau) \le \liminf_{i \to \infty} | \left[ A_i, \tau \right] |_{\Phi}.
\]

\item[$(2)$] If $k_\Phi(\tau)$ is finite, there exists an increasing sequence $\{A_i\}$ in $\mathscr{R}_1^{+}(H)$ converging to $I$ 
satsifying  
\[
\lim_{i \to \infty} | \left[A_i, \tau \right] |_{\Phi} = k_{\Phi}(\tau).
\]

\item[$(3)$] 
Let $\tau^{(j)} \in \mathscr{L}(H)^{n}$ for $j\in \N$. We have
\begin{equation*}
     \max_{j = 1, 2} k_{\Phi}(\tau^{(j)}) \le k_{\Phi}(\tau^{(1)} \oplus \tau^{(2)}) \le k_{\Phi}(\tau^{(1)}) + k_{\Phi}(\tau^{(2)})
\end{equation*}
and 
\begin{equation*}
    k_{\Phi}\left(\bigoplus_{j = 1}^{\infty} \tau^{(j)}\right) = \lim_{m \to \infty} k_{\Phi}\left(\bigoplus_{j = 1}^{m} \tau^{(j)}\right).
\end{equation*}
\end{itemize}
\end{thm}

The following properties can be shown by the averaging trick (see \cite[Proof of Proposition 1.6]{V79}), 

\begin{thm}\label{averaging}
Let the notation be as above. 
Then the following hold:
\begin{itemize}
\item[$(1)$] For $m\in \N$, there exists an increasing sequence $\{A_i\}$ in $\mathscr{R}^+_1(H)$ 
converging to $I$ and satisfying  
    \[
    \lim_{i \to \infty} \NrmPhi{\Com{A_i \otimes I_m}{\tau \otimes I_m}} = k_\Phi(\tau \otimes I_m),
    \]
where $I_m$ is the identity operator of $\C^m$. 
\item[$(2)$] Let $\tau^{(j)}\in \mathscr{L}(H)^{n}$ for $j=1,2,\ldots,m$. 
Then there exists an increasing sequence $\{A_i\}_{i=1}^\infty$ in $\mathscr{R}^+_1(H)^{\oplus m} \subset \mathscr{R}^+_1(H^{\oplus m})$
converging to $I$ and satisfying 
    \[
    \lim_{i \to \infty} |[A_i,\tau^{(1)}\oplus \tau^{(2)}\oplus \cdots \oplus \tau^{(m)}]|_\Phi = 
    k_\Phi(\tau^{(1)}\oplus \tau^{(2)}\oplus \cdots \oplus \tau^{(m)}).
    \]
\item[$(3)$] Let $\tau^{(j)} \in \mathscr{L}(H)^{n}$ and $\lambda^{(j)} \in (\mathbb{C}I)^{n}$ for $j = 1, 2, \ldots, m$. Then we have 
\[
k_{\Phi}(\tau^{(1)} \oplus \cdots \oplus \tau^{(m)}) = k_{\Phi}((\tau^{(1)} - \lambda^{(1)})  \oplus \cdots \oplus (\tau^{(m)} - \lambda^{(m)})).
\]
\end{itemize}
\end{thm}


\section{Main result}
Our basic references for self-similar sets are \cite{Fal} and \cite{Ki}. 
Throughout this section, we treat $\R^n$ as a metric space equipped with the Euclidean metric. 
For $X\subset \R^n$, we denote by $\dim_HX$ the Hausdorff dimension of $X$. 
We denote by $O(n)$ the real orthogonal group. 

We say that a map $F:\R^n\to \R^n$ is a contraction if there exists a constant $0<\lambda<1$ such that for all $x,y\in \R^n$ the inequality 
$$|F(x)-F(y)|\leq \lambda |x-y|$$
holds. 
If moreover equality holds for all $x,y\in \R^n$, we say that $F$ is a \textit{similitude} and $\lambda$ is the \textit{contraction ratio} of $F$.

Let $F_1, F_2, \ldots , F_m$ be contractions of $\mathbb{R}^n$. 
Then it is known that there exists a unique compact set $K \subset \mathbb{R}^n$ satisfying 
\[
K = \bigcup_{j = 1}^m F_j(K),
\]
which we call the self-similar set for $\{F_j\}_{j=1}^m$ (\cite[Theorem 8.3]{Fal}, \cite[Theorem 1.1.4]{Ki}). 

We say that the open set condition holds for $\{F_j\}_{j = 1}^m$ if there exists a bounded non-empty open set $V\subset \R^n$ such that 
\[
\coprod_{j = 1}^m F_j(V) \subset V
\]
with this union disjoint. 

Assume that similitudes $\{F_j\}_{j = 1}^m$ satisfy the open set condition, 
and let $0 < \lambda_j < 1$ be the contraction ratio of $F_j$ for $j = 1, 2, \ldots , m$. 
It follows from \cite[Theorem 8.6]{Fal},\cite[Corollary 1.5.9]{Ki} that the Hausdorff dimension of $K$ is the number $p > 0$ determined by 
$$\sum_{j = 1}^m \lambda_j^p = 1,$$
and the Hausdorff measure $\cH_p(K)$ of $K$ is a non-zero finite number.  
Let $\cW_l=\{1,2,\ldots,m\}^l$, and let $\cW=\bigcup_{l=0}^\infty \cW_l$. 
For $w\in \cW$, we denote its word length by $|w|$. 
For $w \in \cW_l$, we define $F_w = F_{w_1}\circ F_{w_2}\circ\cdots\circ F_{w_l}$, $K_w = F_w(K)$, and 
$$\lambda_w=\lambda_{w_1}\lambda_{w_2}\cdots \lambda_{w_l}.$$
Then we have 
\[
\cH_p(K_w) = \lambda_{w}^p \cH_p(K).
\]

Now we consider the following 4 conditions for self-similar sets. 

\begin{itemize}
\item[(A1)]  Let $F_1, F_2, \ldots , F_m$ be similitudes in $\mathbb{R}^n$, and let $K$ be the corresponding self-similar set. 
For $j = 1, 2, \ldots , m$, we can express $F_j$ as 
    \[
    F_j(x) = \lambda_j U_j(x) + b_j ,
    \]
    where $0 < \lambda_j < 1$ , $U_j$ is an orthogonal matrix in $O(n)$, and $b_j \in \mathbb{R}^n$. 
We assume that the similitudes $F_1, F_2, \ldots , F_m$ satisfy the open set condition, and the Hausdorff dimension $p$ of $K$ 
is strictly larger than $1$. 
\item[(A2)] In addition to (A1), we assume $\lambda_1=\lambda_2=\cdots =\lambda_m$. 
\item[(A3)] In addition to (A2), we assume $U_1=U_2=\cdots= U_m=I_m$. 
\item[(A4)] In addition to (A3), we assume $F_i(K) \cap F_j(K)=\emptyset$ for any $i\neq j$.
\end{itemize}

Voiculescu showed the following theorem in his recent work: 

\begin{thm}[{\cite[Theorem 5.1]{V2021}}] \label{A4}
Let $\tau$ be an n-tuple of commuting Hermitian operators whose spectrum $\sigma(\tau)$ 
is a subset of a self-similar set $K$ satisfying (A4). 
Then there exists a constant $0<\gamma_K<\infty$, depending only on $K$, 
satisfying 
$$k_p^-(\tau)^p=\gamma_K\int_K\fm(x)d\cH_p(x),$$
where $\fm$ is the multiplicity function of the $\cH_p$-absolutely continuous part of $\tau$.  
\end{thm}

\begin{rem}\label{Sierpinski} As we can see from our arguments below, the same statement holds even if the assumption (A4) is relaxed to (A3). 
For example, the statement is true for the Sierpinski gasket and the Sierpinski carpet.  
\end{rem}

The following is our main theorem, where we assume only (A1), but our conclusion requires two constants. 

\begin{thm}\label{main} 
Let $\tau \in \mathscr{L}(H)^n$ be an $n$-tauple of commuting hermitian operators whose spectrum $\sigma(\tau)$ 
is a subset of a self-similar set $K$ satisfying (A1). 
Then there exist two positive constants $C_1, C_2$ determined by only $F_1, F_2, \ldots , F_m$ such that 
    \[
    C_1\int_K \fm\ d \cH_p \le k_p^-(\tau)^p \le C_2 \int_K \fm\ d \cH_p,
    \]
    where $\fm$ is the multiplicity function of the $\cH_p$-absolutely continuous part of $\tau$ 
\end{thm}

We prove Theorem \ref{main} in several steps following Voiculecu's strategy in \cite{V2021} with necessary modifications. 

In the sequel, we assume that $K$ and $\tau$ satisfy the assumption of Theorem \ref{main}. 
We denote by $E(\tau;\cdot)$ the the spectral measure of $\tau$. 
We denote by $\diam(K_w)$ the diameter of $K_w$. 
We set $\lambda_* = \min\{\lambda_j\}_{j=1}^m$. 
We say that two Borel subsets $A$ and $B$ of $K$ are \textit{virtually disjoint} if $\dim_H A\cap B<p$.  
For virtually disjoint $A$ and $B$, we have $\cH_p(A\cap B)=0$.

The first step immediately follows from {\cite[Corollary 4.7]{DV}}. 

\begin{lem}
\label{non-zero k_p^-}
We denote by $\tau_K$ the n-tuple of multiplication operators of coordinate functions in $\mathbb{R}^n$ acting on $L^2(K, \cH_p)$. 
Then we have $k_p^-(\tau_K) > 0$.
\end{lem}

Next we prove \cite[Lemma 5.1]{V2021} under the assumption of (A1). 

\begin{lem}\label{upper} Assume that $\tau$ has a cyclic vector $\xi$. 
Then for some constant $C>0$ depending only on $K$, 
$$k_p^{-}(\tau)\leq C\cH_p(\sigma(\tau))^{1/p}.$$
\end{lem}

\begin{proof} 
We enumerate the family of sets $F_j(K)\cap F_k(K)$ for $j\neq k$, and denote them by $L_1,L_2,\ldots,L_{N}$ (some of them can be empty). 
We further enumerate the family of the sets $F_w(L_i)$ for $|w|\geq 1$ and $1\leq i\leq N$, and denote them by $L_{N+1},L_{N+2},\cdots$. 
Note that we have $\dim_H L_i<p$ thanks to \cite[Theorem 3.3]{Mo}. 

We set $H_1=E(\tau; L_1)H$, and inductively define 
$$H_i=E(\tau; L_i)H\cap (\bigcup_{j=1}^{i-1}H_j)^\perp$$ for $i\geq 2$, and 
$$H_0=\left(\bigcup_{i=1}^\infty H_i\right)^\perp.$$
Then $\tau$ is decomposed as 
$$(\tau,H)=\bigoplus_{i=0}^\infty (\tau_i,H_i).$$ 

We claim $k_p^-(\tau_i)=0$ for all $i\geq 1$. 
Indeed, note that we have $\sigma(\tau_i)\subset L_i$ and $\dim_H L_i<p$. 
We choose $r_i>1$ satisfying $\dim_H L_i<r_i<p$. 
Then $\cH_{r_i}(L_i)=0$. 
Since $\tau_i$ is decomposed into its cyclcic components, we get $k_{r_i}(\tau)=0$ thanks to \cite[Theorem 6.1]{X} and  Theorem \ref{basics},(3). 
Since $\cC_{r_j}\subset C_p^-$, the claim follows from \cite[Corollary 2.6]{V79}.  

By Theorem \ref{basics},(3), the claim implies $k_p^-(\tau)=k_p^-(\tau_0)$. 
Therefore to show Lemma \ref{upper}, we may and do assume $\tau=\tau_0$ as we have $\cH_p(\sigma(\tau_0))\leq \cH_p(\sigma(\tau))$. 
 
For $r\in \N$, let $\Omega(r)$ be the set of all $w= (w_1, w_2, \ldots, w_l)\in \cW$ satisfying 
$$\lambda_w \le \frac{1}{r} < \lambda_{(w_1, w_2, \ldots, w_{l-1})}$$
and $\ K_w \cap \sigma(\tau) \ne \emptyset$,  and let  
   \[
   G(r) = \bigcup_{w \in \Omega(r)} K_w.
   \]
Then we have $\sigma(\tau) \subset G(r)$, and 
\[
\frac{\lambda_*}{r}\diam(K) \le \diam(K_{w}) < \frac{1}{r}\diam(K),
\]
for $w\in \Omega(r)$ by definition. 
Note that $K_{w_1}$ and $K_{w_2}$ are virtually disjoint for all $w_1 \neq w_2 \in \Omega(r)$, and 
$E(\tau;K_{w_1}) \perp E(\tau;K_{w_2})$.
For a given $\varepsilon > 0$, there exists $n_0 \in \mathbb{N}$ so that we have
   \[
   n_0 \le r \Rightarrow \cH_p(G(r)) \le \cH_p(\sigma(\tau)) + \varepsilon
   \]
because $\sigma(\tau) = \bigcap_r G(r)$. 
We denote $E_w = E(\tau; K_w)$. 
Then we have
   \[
   \sum_{w \in \Omega(r)} E_w = I.
   \]
    We also have
    \begin{equation*}
    |\Omega(r)| \le \frac{r^p}{\lambda_*^p\cH_p(K)} (\cH_p(\sigma(\tau)) + \varepsilon)
    \end{equation*}
for $r\geq n_0$ as $\sum_{w \in \Omega(r)} \cH_p(K_w) = \cH_p(G(r))$ and  if $w \in \Omega(r)$ we have
    \[
    \left(\frac{\lambda_*}{r}\right)^p \cH_p(K) \le \cH_p(K_w).
    \]

    Let $P_r$ be the orthogonal projection onto the linear span of $\{E_w \xi|\;w \in \Omega(r)\}$. 
We can express $P_r$ as
    \[
    P_r = \sum_{w \in \Omega(r)} e_w \otimes e_w^\ast,
    \]
    where $e_w =\frac{1}{\Nrm{E_w\xi}} E_w\xi$ if $E_w\xi\neq 0$ and $e_w=0$ if $E_w\xi=0$. 
Since $E_w$ commutes with $\tau$ and $P_r$, we have  
    \begin{equation*}
        \begin{split}
            \Nrm{\Com{P_r}{\tau}} &= \max_{w \in \Omega(r)} \Nrm{\Com{P_r}{\tau} E_w }\\
            &=\max_{1\leq i\leq n} \max_{w\in G(r)} \Nrm{e_w \otimes (T_i e_w)^\ast - (T_i e_w) \otimes e_w^\ast}.
        \end{split}
    \end{equation*}
Since $\diam(K_w)\leq \diam(K)/r$, we get $\Nrm{\Com{P_r}{\tau}} \le 2\diam(K)/r$. 

Since $\{P_r\}_{r=1}^\infty$ is an increasing sequence of projections, it strongly converges to a projection, say $P$. 
Since $P_r\xi=\xi$ for $r\in \N$ and $\lim_{r\to\infty}\|[P_r,\tau]\|=0$, we have $P\xi=\xi$ and $\Com{P}{\tau} = 0$. 
Thus we get $P=I$ as $\xi$ is cyclic for $\tau$. 

Since $\mathrm{rank}(P_r) \le |\Omega(r)|$,
\begin{align*}
|\Com{P_r}{\tau}|_p^- &\le \sum_{k = 1}^{2|\Omega(r)|} k^{-1 + 1/p} \Nrm{\Com{P_r}{\tau}} \le 
     \sum_{k = 1}^{2|\Omega(r)|} k^{-1 + 1/p} \frac{2\diam(K)}{r} \\
 &\le p(2|\Omega(r)|)^{1/p} \frac{2\diam(K)}{r} \le C (\cH_p(\sigma(\tau)) + \varepsilon)^{1/p},
\end{align*}
where $C = p \cdot 2^{1 + 1/p}\lambda_*^{-1} \cdot \cH_p(K)^{- 1/p}$. Using Theorem \ref{basics},(1), we get the statement.
\end{proof}

We can prove \cite[Lemma 5.2]{V2021} under the assumption (A1) now.  

\begin{lem}
\label{singular vanishing}
Assume that $\sigma(\tau)\subset K$ and the spectral measure of $\tau$ is singular with respect to the Hausdorff measure $\cH_p$. 
Then $k_p^-(\tau) = 0$.
\end{lem}

\begin{proof} As in the proof of the previous lemma, we may assume that $\tau=\tau_0$.  
Then we can prove the statement exactly in the same way as in the proof of \cite[Lemma 5.2]{V2021} by using Lemma \ref{upper}. 
\end{proof}

The main technical improvement in \cite{V2021} is the ampliation homogeneity \cite[Theorem 3.1]{V2021} showing 
$k_p^-(\tau\otimes I_m)=m^{\frac{1}{p}}k_p^-(\tau)$. 
To deal with the contraction rates $\lambda_j$ depending on $j$, we modify it as follows. 

\begin{lem}
\label{tensor diagonal}
Let $\Phi$ be a symmetric function, let $0 \le c_1 \le a_1, a_2, \ldots , a_m \le c_2$, and let $\tau \in \mathscr{L}(H)^n$. 
Then we have
    \[
    c_1 k_\Phi(\tau \otimes I_m) \le k_\Phi(\tau \otimes \diag(a_i)_{i = 1}^m) \le c_2 k_\Phi(\tau \otimes I_m). 
    \]
In paricular, 
  \[
    c_1m^{\frac{1}{p}} k_p^-(\tau) \le k_p^-((\tau \otimes \diag(a_i)_{i = 1}^m) \le c_2m^{\frac{1}{p}} k_p^-(\tau).    
    \]
    \end{lem}
\begin{proof}

    For $X_1, X_2, \ldots, X_m \in \mathscr{L}(H)$ and $a_1, a_2, \ldots, a_m, b_1, b_2, \ldots, b_m \in \mathbb{R}$ satisfying $0 \le a_j \le b_j$, we denote the singular values of $X_j$ by $\{s_i^{(j)}\}_{i = 1}^\infty$. If we define 
    $A = \bigoplus_{j = 1}^m a_j X_j$ and $B = \bigoplus_{j = 1}^m b_j X_j$, then the singular values of $A$ and $B$ are respectively $\{a_j s_i^{(j)}\}_{i, j}$ and $\{b_j s_i^{(j)}\}_{i, j}$. So for all $n\in \N$ we have
    \[
    \sum_{k = 1}^n \mu(A)_k = \sum_{k = 1}^n a_{j_k} s_{i_k}^{(j_k)} \le \sum_{k = 1}^n b_{j_k} s_{i_k}^{(j_k)} \le \sum_{k = 1}^n \mu(B)_k.
    \]
    Then it follows from \cite[Theorem 1.16 (b)]{S} that $\NrmPhi{A} \le \NrmPhi{B}$.

Thanks to Theorem \ref{averaging},(2), there exists an increasing sequence $\{A_i\}_{i=1}^\infty$ in $\mathscr{R}^+_1(H)^{\oplus m} \subset \mathscr{R}^+_1(H^{\oplus m})$
such that $\{A_i\}$ strongly converges to $I$ and 
    \[
    \lim_{i \to \infty} \NrmPhi{\Com{\tau \otimes \diag(a_j)}{A_i}} = k_\Phi(\tau \otimes \diag(a_j)).
    \]
Expressing $A_i$ as 
$$A_i = A_1^{(i)} \oplus A_2^{(i)} \oplus \cdots A_m^{(i)} \in \mathscr{R}^+_1(H)^{\oplus m},$$
we get  
    \begin{equation*}
        \begin{split}
            & c_1 \NrmPhi{\Com{\tau \otimes I_m}{A_i}} = \NrmPhi{\Com{\bigoplus_{j = 1}^m c_1 \tau}{\bigoplus_{j = 1}^mA_j^{(i)}}} = \NrmPhi{\bigoplus_{j = 1}^m c_1 \Com{\tau}{A_j^{(i)}}}\\
            &\le \NrmPhi{\bigoplus_{j = 1}^m a_j \Com{\tau}{A_j^{(i)}}} = \NrmPhi{\Com{\tau \otimes \diag(a_j)}{A_i}}.
        \end{split}
    \end{equation*}
    Then by Theorem \ref{basics},(1), we get 
    \[
    c_1 k_\Phi(\tau \otimes I_m) \le k_\Phi(\tau \otimes \diag(a_j)).
    \]
By Theorem \ref{averaging},(1), we can take an increasing sequence $\{B_i\} \in \mathscr{R}^+_1(H)$ strongly converges to $I$ and satisfying 
    \[
    \lim_{i \to \infty} \NrmPhi{\Com{\tau \otimes I_m}{B_i \otimes I_m}} = k_\Phi(\tau \otimes I_m).
    \]
Then we have
    \begin{equation*}
        \begin{split}
               &\NrmPhi{\Com{\tau \otimes \diag(a_j)}{B_i \otimes I_m}} = \NrmPhi{\bigoplus_{j = 1}^m a_j \Com{\tau}{B_i}}\\
               &\le \NrmPhi{\bigoplus_{j = 1}^m c_2 \Com{\tau}{B_i}} = c_2\NrmPhi{\Com{\tau \otimes I_m}{B_i \otimes I_m}} .
        \end{split}
    \end{equation*}
    Using Theorem \ref{basics},(1), we obtain
$$k_\Phi(\tau \otimes \diag(a_j)) \le c_2 k_\Phi(\tau \otimes I_m).\eqno \qedhere$$
\end{proof}

To deal with the rotation part in $F_j$, we use a modified quasicentral modulus, 
as suggested in \cite[Remark 6.2]{V2021}.
 
\begin{dfn}[{\cite[Theorem 1.7]{V79}}]
    For a symmetric function $\Phi$ and $\tau \in \mathscr{L}(H)^n$, we define $\widetilde{k_\Phi}(\tau)$ by 
    \[
    \widetilde{k_\Phi}(\tau) = \liminf_{A \in \mathscr{R}^+_1} \NrmPhib{\Com{A}{\tau}}, 
    \]
    where $\NrmPhib{\tau} = \NrmPhi{\left(\sum_{i = 1}^n T_i^\ast T_i\right)^{\frac{1}{2}}}$.
\end{dfn}
Let $e_1, e_2, \ldots, e_n$ be the orthonormal basis of $\mathbb{C}^n$ and define $d_\tau : H \to H \otimes \mathbb{C}^n$ by
\[
d_\tau(\xi) = \sum_{i = 1}^n T_i \xi \otimes e_i
\]
Then we have $d_\tau^\ast d_\tau = \sum_{i = 1}^n T_i^\ast T_i$ and $\NrmPhib{\tau} = \NrmPhi{d_\tau}$. 
Since 
$$|\tau|_\Phi\leq |d_\tau|_\Phi\leq n|\tau|_\Phi,$$
we have 
$$k_\Phi(\tau)\leq \widetilde{k_\Phi}(\tau)\leq n k_\Phi(\tau).$$
$\widetilde{k_\Phi}$ and $k_\Phi$ share similar properties such as Theorem \ref{basics}, Theorem \ref{averaging}, 
Lemma \ref{singular vanishing}, and Lemma \ref{tensor diagonal} etc. However, $\widetilde{k_\Phi}$ is distinguished from $k_\Phi$ in that $\widetilde{k_\Phi}$ has the following property.

For $U = (u_{i,j})_{i,j} \in O(n)$ and $\tau = (T_i)_i \in \mathscr{L}(H)^n$, 
we denote 
    \[
    U \tau = \left(\sum_{i = 1}^n u_{1,i} T_i, \sum_{i = 1}^n u_{2,i} T_i, \ldots, \sum_{i = 1}^n u_{n, i} T_i\right) \in \mathscr{L}(H)^n.
    \]

\begin{lem}
\label{solution for O(n)}
    Let $\Phi$ be a symmetric function, let $\tau^{(1)}, \tau^{(2)}, \ldots , \tau^{(m)} \in \mathscr{L}(H)^n$, and let $U_1, U_2, \ldots, U_m \in O(n)$. 
    Then we have 
    \[
    \widetilde{k_{\Phi}}\left(\bigoplus_{j = 1}^m \tau^{(j)}\right) = \widetilde{k_{\Phi}}\left(\bigoplus_{j = 1}^m U_j \tau^{(j)}\right).
    \]
\end{lem}
\begin{proof}
    By symmetry, it suffices to show 
    \[
  \widetilde{k_\Phi}\left(\bigoplus_{j = 1}^m U_j \tau^{(j)}\right) \leq  \widetilde{k_\Phi}\left(\bigoplus_{j = 1}^m \tau^{(j)}\right).
    \] 
    
    From Theorem \ref{averaging},(2), there exists an increasing sequence $\{A_i\}$ in the set $\mathscr{R}^+_1(H)^{(m)}$ converging to $I$ and 
    \[
    \lim_{i \to \infty} \NrmPhib{\Com{A_i}{\bigoplus_{j=1}^m \tau^{(j)}} } = \widetilde{k_\Phi}\left(\bigoplus_{j = 1}^m \tau^{(j)}\right).
    \]
Denoting 
    $$A_j=A_j^{(1)}\oplus A_j^{(2)}\oplus \cdots \oplus A_j^{(n)},$$
we have
    \begin{equation*}
        \begin{split}
&\NrmPhib{\Com{A_i}{\bigoplus_{j = 1}^m \tau^{(j)}}} = \NrmPhib{\bigoplus_{j = 1}^m \Com{A_i^{(j)}}{\tau^{(j)}} } 
= \NrmPhib{\bigoplus_{j = 1}^m U_j\Com{A_i^{(j)}}{\tau^{(j)}} }\\
&= \NrmPhib{\bigoplus_{j = 1}^m \Com{A_i^{(j)}}{U_j \tau^{(j)}} } = \NrmPhib{\Com{A_i}{\bigoplus_{j = 1}^m U_j\tau^{(j)}} }.
        \end{split}
    \end{equation*}
    Using Theorem \ref{basics},(1), we get 
$$
    \widetilde{k_\Phi}\left(\bigoplus_{j = 1}^m \tau^{(j)}\right)=\lim_{i\to\infty} \NrmPhib{\Com{A_i}{\bigoplus_{j = 1}^m U_j\tau^{(j)}} }
     \geq \widetilde{k_\Phi}\left(\bigoplus_{j = 1}^m U_j \tau_j\right) .
\eqno \qedhere    $$
\end{proof}

Note that Lemma \ref{non-zero k_p^-} and Lemma \ref{upper} show $0<\widetilde{k_p^-}(\tau_K)<\infty$. 
We set 
$$\widetilde{\gamma_K}=\frac{\widetilde{k_p^-}(\tau_K)^p}{\cH_p(K)}.$$ 
Now Theorem \ref{main} follows from the following theorem. 

\begin{thm}\label{tinequality} 
Under the assumption of Theorem \ref{main}, the following holds:
$$ \lambda_*^p\widetilde{\gamma_{K}}\int_K \fm\ d \cH_p \le \widetilde{k_p^-}(\tau)^p \le \lambda_*^{-p}\widetilde{\gamma_{K}} \int_K \fm\ d \cH_p.$$
\end{thm}

\begin{proof}
We denote by $\tau_{p,a}$ and $\tau_{p,s}$ respectively the $\cH_p$-absolutely continuous part of $\tau$ and the $\cH_p$-singular part of $\tau$. 
Thanks to Lemma \ref{singular vanishing}, we have $\widetilde{k_p^-}(\tau_{p,s})=0$. 
Thus Theorem \ref{basics},(3) shows that to prove the statement, we may and do assume $\tau=\tau_{p,a}$, that is, the spectral measure 
of $\tau$ is absolutely continuous with respect to $\cH_p$. 

Up to unitary equivalence, we may assume  
\[\tau =\bigoplus_{k = 1, 2, \ldots , \infty} \tau_{X_k}^{\oplus k},\]
where $\{X_k\}_{k=1}^\infty$ are mutually disjoint Borel subsets of $K$, and $\tau_{X_k}$ is the $n$-tuple of multiplication operators 
by coordinate functions in $\R^n$ acting on $L^2(X_k,\cH_p)$.
We have $\fm(x) = k$ if and only if $x \in X_k$.

First, we prove the statement in the case of multiplicity $1$, that is, $\tau = \tau_X$ for a Borel subset $X \subset K$.
In the same way as in the proof of \cite[Theorem 5.1]{V2021}, this case is reduced to the case where $X$ is a finite union of $K_{w}$. 
Thus we assume $X = \bigcup_{j = 1}^l K_{w_j}$ with
    \[
    F_{w_j}(x) = r_j V_j(x) + a_j,
    \]
and that $\{K_{w_j}\}_{j=1}^l$ are mutually virtually disjoint. 
If necessary, we decompose each $K_{w_j}$ and assume that 
    \[
    \lambda_* r_1 \le r_j \le r_1
    \]
    for $j = 1, 2, \ldots, m$. 
 We have a unitary $\varphi$ from $L^2(F_{w_j}(X), \cH_p)$ to $L^2(X, \cH_p)$ defined by
    \[
    \varphi(f) = r_j^{\frac{p}{2}} f \circ F_{w_j},
    \]
and $\tau_{K_{w_j}} = \tau_{F_{w_j}(K)}$ is unitarily equivalent to $F_{w_j}(\tau_K)$, 
where for 
$$F = (F_1, F_2, \ldots, F_n) : \mathbb{R}^n \to \mathbb{R}^n$$
and $\tau = (T_1, T_2, \ldots, T_n) \in \mathscr{L}(H)^n$,
    \[
    F(\tau) = (F_1(T_1, T_2, \ldots, T_n), F_2(T_1, T_2, \ldots, T_n), \ldots, F_n(T_1, T_2, \ldots, T_n)).
    \]
    Then we can see that 
    \begin{equation*}
        \begin{split}
            &\kpb{\tau_X} = \kpb{\bigoplus_{j = 1}^l \tau_{K_{w_j}}} = \kpb{\bigoplus_{j = 1}^l F_{w_j}(\tau_K)}\\ 
	    &= \kpb{\bigoplus_{j = 1}^l \left(r_j V_j(\tau_K) + a_j\right)}\\
            &= \kpb{\bigoplus_{j = 1}^l r_j V_j(\tau_K)} \ \text{(by Theorem \ref{averaging},(3))}\\
            &= \kpb{\bigoplus_{j = 1}^l r_j \tau_K}  \ \text{(by Theorem \ref{solution for O(n)})}\\
            &= \kpb{\tau_K \otimes \diag(r_j)_{j = 1}^l}.
        \end{split}
    \end{equation*}
By Lemma \ref{tensor diagonal}, we get,
    \[
    \lambda_\ast r_1 l^{1/p}\kpb{\tau_K} \le \kpb{\tau_K \otimes \diag(r_j)_{j = 1}^l} \le r_1 l^{1/p}\kpb{\tau_K},
    \]
and on the other hand, 
    \[
    \lambda_\ast^p r_1^p l \cH_p(K) \le \cH_p(X) \le r_1^p l\cH_p(K). 
    \]
Therefore we obtain the statement. 

Next we show the statement in the case of finite multiplicity. 
In this case, we may assume $\tau = \bigoplus_{j = 1}^l \tau_{Y_j}$ with $Y_j \subset K$. 
We choose a sufficiently large number $e$ and compose $F_1$ and $F_2$ $e$ times each to make similitudes 
$F_{w_1}, F_{w_2}, \ldots, F_{w_m}$ so that $\{F_{w_j}(X_j)\}_{j=1}^l$ are virtually disjoint. 
Now we can express 
    \[
    F_{w_j}^{-1}(x) = r W_j(x) + c_j 
    \]
    with $r = \lambda_1^{-e} \lambda_2^{-e}>0$ and some $W_j \in O(n)$, $c_j \in \mathbb{R}^n$. 
Then we have
\begin{align*}
\lefteqn{\kpb{\tau} = \kpb{\bigoplus_{j = 1}^l \tau_{Y_j}} = \kpb{\bigoplus_{j = 1}^l F_{w_j}^{-1} F_{w_j} (\tau_{Y_j})} } \\
 &= \kpb{\bigoplus_{j = 1}^l \left(r W_j(F_{w_j}(\tau_{w_j})) + c_j\right)}=r \kpb{\bigoplus_{j = 1}^l F_{w_j}(\tau_{w_j})} \\
 &= r \kpb{\bigoplus_{j = 1}^l \tau_{F_{w_j}(Y_j)}}= r \kpb{\tau_{Y}},
\end{align*}
where $Y= \bigcup_{j = 1}^l F_{w_j}(Y_j)$. 
On the other hand, 
    \[
    r^p \cH_p(Y) = \sum_{j = 1}^l r^p \cH_p(F_{w_j}(Y_j)) = \sum_{j = 1}^l \cH_p(Y_j) = \int_K \fm d\cH_p.
    \]

The general case is reduced to the finite multiplicity case thanks to Theorem \ref{basics},(3), and we are done. 
\end{proof}

\begin{cor}
    Let $\tau$ and $K$ be as above. Then the following conditions are equivalent:
    \begin{enumerate}
        \item[$(1)$] $k_p^-(\tau) = 0$.
        \item[$(2)$] The spectral measure of $\tau$ is singular with respect to $\cH_p$.
    \end{enumerate}
\end{cor}

An easy modification of the proof of Theorem \ref{tinequality} shows the following: 

\begin{cor}\label{tequality} In addition to the assumption of Theorem \ref{main}, we assume (A2). 
Then the following holds:
$$\widetilde{k_p^-}(\tau)^p =\widetilde{\gamma_{K}} \int_K \fm\ d \cH_p.$$
\end{cor}

\begin{ques} Let $\tau\in \mathscr{L}(H)^n$, let $1<p<\infty$, and let $a=(a_j)\in \R^k$. 
In view of the proof of Theorem \ref{tinequality}, we can see that if 
\begin{equation}\label{lp1}
\widetilde{k_p^-}(\tau\otimes \diag\left(a_{j}\right)_{j=1}^{k})=\widetilde{k_p^{-}}(\tau)\|a\|_p
\end{equation}
is true, the formula in Corollary \ref{tequality} would hold under the condition (A1). 
As was already asked in \cite[Remark 6.1]{V2021}, a more fundamental question is:  
does the equality 
\begin{equation}\label{lp2}
\widetilde{k_p^-}(\tau_1\oplus \tau_2)^p=\widetilde{k_p^-}(\tau_1)^p+\widetilde{k_p^-}(\tau_2)^p
\end{equation}
hold in general?

Thanks to \cite[Lemma 3.2]{V2021}, one might hope to prove (\ref{lp2}) (or (\ref{lp1})) by reducing it to the following question. 
Let $\{S_l\}$ and $\{T_m\}$ be sequences of finite rank operators satisfying 
$$\lim_{l\to \infty}\|S_l\|=\lim_{m\to\infty}\|T_m\|=0,$$ 
$$\lim_{l\to \infty}|S_l|_p^-=\alpha,\quad \lim_{m\to\infty}|T_m|_p^-=\beta.$$
Then does the equality 
\begin{equation}\label{lp3}
\inf_{l_r,m_r}\liminf_{r \to \infty}|S_{l_r}\oplus T_{m_r}|_p^-=(\alpha^p+\beta^p)^{1/p},
\end{equation}
where the infimum is taken over all subsequences, hold true ?  
However, there exists an easy counter example to (\ref{lp3}). 
Let $P_l$ be a projection of rank $l$, and let $S_l=p^{-1}2^{-l/p}P_{2^l}$ and $T_m=3^{1/p}S_m$. 
Then it is straightforward to show that (\ref{lp3}) does not hold for $\{S_l\}$ and $\{T_m\}$, 
and the left-hand side is strictly larger than the right-hand side. 
This means that we should use, not only subsequences, but also convex combinations of 
$\{S_l\}$ and $\{T_m\}$ in the left-hand side of (\ref{lp3}) to raise a reasonable question. 
\end{ques}

\begin{ques} For our purpose, we can relax (\ref{lp1}) as follows. 
Let $\{N_n\}_{n=1}^\infty$ be an increasing sequence of natural numbers converging to infinity, 
and let $a_{n,j}$, $n\in \N$, $1\leq j\leq N_n$, be positive numbers such that there exists $c>0$ satisfying 
$$\sum_{j=1}^{N_n}|a_{n,j}|^p=1,$$
$$\max\{a_{n,j}\}_{j=1}^{N_n}\leq c\min\{a_{n,j}\}_{j=1}^{N_n},$$
for all $n\in \N$. 
Then we can see that if 
\begin{equation}\label{lp4}
\lim_{n\to \infty}\widetilde{k_p^-}(\tau\otimes \diag\left(a_{n,j}\right)_{j=1}^{N_n})=\widetilde{k_p^{-}}(\tau)
\end{equation}
is true, the formula in Corollary \ref{tequality} would hold under the condition (A1). 
It might be interesting to ask what kind of condition on $a_{n,j}$ assures (\ref{lp4}).  
\end{ques}

\section*{Acknowledgement} 
The authors would like to thank Dan Voiculescu for useful discussion, and Kenneth Falconer and Jun Kigami for having informed 
the authors of the reference \cite{Mo}.  

\section*{Declarations} 
M. I. is supported in part by JSPS KAKENHI Grant Number JP20H01805.
The authors have no relevant financial or non-financial interests to disclose. 
Data sharing not applicable to this article as no datasets were generated or analysed during the current study.


\begin{thebibliography}{1}

\bibitem{BV} Bercovici, Hari and Voiculescu, Dan. 
\textit{The analogue of Kuroda's theorem for n-tuples.} 
The Gohberg anniversary collection, Vol. II (Calgary, AB, 1988), 57–60, Oper. Theory Adv. Appl., 41, Birkh\"{a}user, Basel, 1989.

\bibitem{DV} David, Guy and Voiculescu, Dan. 
\textit{s-Numbers of singular integrals for the invariance of absolutely continuous spectra in fractional dimensions.} 
Journal of Functional Analysis \textbf{94} (1990), no. 1, 14--26.


\bibitem{Fal} Falconer, Kenneth. 
\textit{The geometry of fractal sets.} 
Cambridge Tracts in Mathematics, 85. Cambridge University Press, Cambridge, 1986.

\bibitem{GK} Gohberg, Israel and Krein, Mark. 
\textit{Introduction to the Theory of Linear non-selfadjoint operators.} 
Translated from the Russian by A. Feinstein Translations of Mathematical Monographs, Vol. 18 American Mathematical Society, Providence, R.I. 1969.


\bibitem{Ki} Kigami, Jun. 
\textit{Analysis on fractals.} 
Cambridge Tracts in Mathematics, 143. Cambridge University Press, Cambridge, 2001.

\bibitem{Mo} Mor\'{a}n, Manuel. 
\textit{Dynamical boundary of a self-similar set.} 
Fund. Math. \textbf{160} (1999), no. 1, 1--14. 

\bibitem{S} Simon, Barry. 
\textit{Trace ideals and their applications.} 
Mathematical Surveys and Monographs, 120. American Mathematical Society, Providence, RI, 2005. 


\bibitem{V76} Voiculescu, Dan. 
\textit{A non-commutative Weyl-von Neumann theorem.} 
Rev. Roumaine Math. Pures Appl. \textbf{21} (1976), no. 1, 97--113.

\bibitem{V79} Voiculescu, Dan. 
\textit{Some results on norm-ideal perturbations of Hilbert space operators.}  
Journal of Operator Theory \textbf{2} (1979), no. 1,  3--37.

\bibitem{V81} Voiculescu, Dan. 
\textit{Some results on norm-ideal perturbations of Hilbert space operators. II.}  
Journal of Operator Theory \textbf{5} (1981), no. 1, 77--100.

\bibitem{V84} Voiculescu, Dan. 
\textit{Hilbert space operators modulo normed ideals.} 
Proceedings of the International Congress of Mathematicians, Vol. 1, 2 (Warsaw, 1983), 1041--1047, PWN, Warsaw, 1984. 

\bibitem{V90} Voiculescu, Dan. 
\textit{On the existence of quasicentral approximate units relative to normed ideals. Part I.}  
Journal of Functional Analysis \textbf{91} (1990), no. 1, 1--36.


\bibitem{V2021} Voiculescu, Dan. 
\textit{The formula for the quasicentral modulus in the case of spectral measures on fractals.}  
Journal of Fractal Geometry \textbf{8} (2021), no. 4, 347--361.


\bibitem{X} Xia, Jingbo. 
\textit{Diagonalization modulo norm ideals with Lipschitz estimates.} J. Funct. Anal. \textbf{145} (1997), no. 2, 491--526.
\end{thebibliography}
\end{document}